\documentclass[11pt]{amsart}
\usepackage{amsmath}
\usepackage{amssymb, verbatim}
\usepackage{pstricks,pst-node,pst-plot}

\title{An extension theorem for K\"ahler currents}
\author[T.C. Collins]{Tristan C. Collins}
\address{Department of Mathematics, Columbia University, 2990 Broadway, New York, NY 10027}
\email{tcollins@math.columbia.edu}
\author[V. Tosatti]{Valentino Tosatti$^{*}$}
\thanks{$^{*}$Supported in part by a Sloan Research Fellowship and NSF grants DMS-1236969 and DMS-1308988.}
\address{Department of Mathematics, Northwestern University, 2033 Sheridan Road, Evanston, IL 60201}
\email{tosatti@math.northwestern.edu}

\theoremstyle{plain}
\newtheorem{thm}{Theorem}[section]

\newtheorem{lem}[thm]{Lemma}

\theoremstyle{definition}
\newtheorem{ex}[thm]{Example}

\numberwithin{equation}{section}

\newcommand{\del}{\partial}
\newcommand{\de}{\partial}
\newcommand{\dbar}{\overline{\del}}
\newcommand{\ddb}{i\del\dbar}

\newcommand{\ve}{\varepsilon}
\newcommand{\vp}{\varphi}
\newcommand{\ti}[1]{\tilde{#1}}

\renewcommand{\leq}{\leqslant}
\renewcommand{\geq}{\geqslant}

\renewcommand{\epsilon}{\varepsilon}
\renewcommand{\phi}{\varphi}
\newcommand{\ov}[1]{\overline{#1}}

\begin{document}
\begin{abstract}
We prove an extension theorem for K\"ahler currents with analytic singularities in a K\"ahler class on a complex submanifold of a compact K\"ahler manifold.
\end{abstract}
\maketitle
\section{Introduction}
The problem that we consider in this paper is the following: given a compact K\"ahler manifold $(X,\omega)$, a compact complex submanifold $V\subset X$, and a closed positive current $T$ on $V$ in the class $[\omega|_V]$, can we find a closed positive current $\ti{T}$ on $X$ in the class $[\omega]$ with $T=\ti{T}|_V$? Extension questions like this have recently generated a great deal of interest thanks to their analytic and geometric applications \cite{CGZ, His, OV, Sch, Wu}.
The first result in this direction is due to Schumacher \cite{Sch} who proved that if $[\omega]$ is rational (hence $X$ is projective), then any smooth K\"ahler metric on $V$ in the class $[\omega|_V]$ extends to a smooth K\"ahler metric on $X$ in the class $[\omega]$ (see also \cite{CGZ, OV, Sa, Wu}). More recently,  Coman-Guedj-Zeriahi proved in \cite{CGZ} that under the same rationality assumption, every closed positive current $T$ on a closed analytic subvariety $V\subset X$ in the class $[\omega|_V]$ extends to $X$.

In our main theorem we get rid of rationality/projectivity assumptions in the case of extension of K\"ahler currents with analytic singularities from a submanifold. More precisely, we prove:

\begin{thm}\label{extend}
Let $(X,\omega)$ be a compact K\"ahler manifold and let $V \subset X$ be a positive-dimensional compact complex submanifold.  Let $T$ be a K\"ahler current with analytic singularities on $V$ in the K\"ahler class $[\omega|_{V}]$.  Then there exists a K\"ahler current $\ti{T}$ on $X$ in the class  $[\omega]$ with $T = \ti{T}|_{V}$.
\end{thm}

The techniques that have been used in the past to approach this type of extension problems range from Siu's Stein neighborhood theorem \cite{Si}, to results of Coltoiu \cite{Co} on extending Runge subsets of analytic subsets of $\mathbb{C}^{N}$, to the Ohsawa-Takegoshi $L^2$ extension theorem \cite{OT}. In this paper we introduce a new constructive extension technique which uses resolution of singularities to obtain estimates which allow us to glue plurisubharmonic functions with analytic singularities near their polar set by a modification of a classical argument of Richberg \cite{Rich}. The ideas of using resolution of singularities comes from the recent work of Collins-Greenleaf-Pramanik \cite{CGP} on sharp estimates for singular integral operators, which was motivated by the seminal work of Phong, Stein and Sturm \cite{PSS, PS, PS2}. One advantage of these local techniques is that we can work on general K\"ahler manifolds with arbitrary K\"ahler classes, although at present we can only extend K\"ahler currents with analytic singularities from smooth subvarieties.  It would be very interesting to know how far these techniques can be pushed, but we have not undertaken this here.

In our recent work \cite{CT}, we dealt with a very similar extension problem, and employed a related argument of a more global flavor where we used resolution of singularities from the very beginning and worked on a blowup of $X$. This approach is technically simpler, because it allows us to use directly Richberg's gluing technique, but it seems not to be strong enough to prove Theorem \ref{extend}, because the extension is achieved only on a blowup of $X$.
On the other hand, the proof of Theorem \ref{extend} can easily be adapted to give another proof of \cite[Theorem 3.2]{CT}, which is the key technical result needed in that paper.\\

{\bf Acknowledgment. }We are grateful to the referee for some useful corrections and suggestions. The second named author was supported in part by a Sloan Research Fellowship and NSF grants DMS-1236969 and DMS-1308988.

\section{Proof of Theorem \ref{extend}}\label{sectproof}

Let us first recall the necessary definitions and concepts needed for the proof of Theorem \ref{extend}. Let $(X,\omega)$ be a compact K\"ahler manifold. We say that a closed positive $(1,1)$ current $T$ on $X$ is a K\"ahler current if it satisfies $T\geq\ve\omega$ as currents, for some $\ve>0$.

Every closed positive $(1,1)$ current $T$ cohomologous to $\omega$ is of the form $T=\omega+\ddb\vp$ for some quasi-plurisubharmonic function $\vp$.  We say that $\vp$ (or $T$) has analytic singularities if for any given $x\in X$ there is a neighborhood $U$ of $x$ with holomorphic functions $f_1,\dots,f_N$ and a smooth function $g$ such that on $U$ we have
$$\vp=c\log\left(\sum_i |f_i|^2\right)+g,$$
for some $c>0$. In this case the polar set $\{\vp=-\infty\}$ is a closed analytic subvariety of $X$. A basic result of Demailly \cite{Dem92} implies that every K\"ahler current can be approximated (in the weak topology) by K\"ahler currents with analytic singularities in the same cohomology class.

We now turn to the proof of Theorem \ref{extend}. The proof is somewhat technical, so we briefly outline the main steps.  First,
we construct local extensions on a finite open covering of $V$.  We then use an idea of Richberg \cite{Rich} to
glue the local potentials to construct a K\"ahler current in a \emph{pinched} tubular neighborhood in $X$.  That is,
near any point where $T$ has positive Lelong number, the diameter of the tubular neighborhood might
pinch to zero. By adding a suitable potential to the K\"ahler metric, we also obtain a global K\"ahler current $R$ on $X$
with analytic singularities along $V$. We then use a resolution of singularities argument to show that, after possibly reducing the Lelong numbers
of the global K\"ahler current $R$, we can glue the K\"ahler current on the pinched neighborhood to $R$,
to obtain a global K\"ahler current.

\begin{proof}[Proof of Theorem \ref{extend}]
By assumption, we can write $T=\omega|_V+\ddb\phi$ for some quasi-plurisubharmonic function $\vp$ on $V$ which has analytic singularities on $V$. Our goal is to extend $\vp$ to
a function $\Phi$ on $X$ with $\omega+\ddb\Phi$ a K\"ahler current.

Thanks to \cite[Lemma 2.1]{DP}, there exists a function $\psi:X \rightarrow [-\infty, \infty)$ which is smooth on $X\backslash V$,
with analytic singularities along $V$, and with $\ddb\psi\geq -A\omega$ as currents on $X$, for some large $A$. Then, for $\delta$ sufficiently small, we have that
$R = \omega+\delta\ddb \psi$ is a K\"ahler current on $X$ with analytic singularities along $V$. Fix one such $\delta$ and define $F=\delta\psi$.

Choose $\ve>0$ small enough so that $$T= \omega|_{V} +\ddb\phi \geq 3\epsilon \omega|_{V},$$ holds as currents on $V$.
We can cover $V$ by finitely many charts $\{W_j\}_{1\leq j\leq N}$ such that on each $W_j$ there are local coordinates
$(z_1,\dots,z_n)$ and so that $V\cap W_j=\{z_1=\dots=z_{n-k}=0\}$, where $k=\dim V$.
Write $z=(z_1,\dots,z_{n-k})$ and $z'=(z_{n-k+1},\dots,z_n)$
and define a function $\vp_j$ on $W_j$ (with analytic singularities) by
$$\vp_j(z,z')=\vp(z')+A|z|^2,$$
where $A>0$ is a constant. If we shrink the $W_j$'s slightly, still preserving the property that $V\subset\cup_j W_j$, we can choose $A$ sufficiently large so that
$$\omega+\ddb\vp_j\geq 2\ve\omega,$$
holds on $W_j$ for all $j$. It will also be useful to fix slightly smaller open sets $W_j'\Subset U_j\Subset W_j$ such that $\cup_j W_j'$ still covers $V$.
Note that since $\vp$ is smooth at the generic point of $V$, by construction all functions $\vp_j$ are also smooth in a neighborhood of the generic point of $V\cap W_j$.

We wish to glue the functions $\vp_j$ together to produce a K\"ahler current defined in a neighborhood of $V$ in $X$.
This would be straightforward if the functions $\vp_j$ were continuous, thanks to a procedure of Richberg \cite{Rich},
but in our case the functions $\vp_j|_V$ have poles along
$$P=E_+(T)=\left(\bigcup_j\{\vp_j=-\infty\}\right)\cap V.$$
We can still use the technique of Richberg to produce a K\"ahler current in an open neighborhood in $X$ which does not contain the polar set $P$.
An argument using resolution of singularities will then allow us to get a K\"ahler current in a whole neighborhood of $V$.

The first step is to consider two open sets $W_1, W_2$ in the covering with $W_1'\cap W_2'\cap V$ nonempty,
and fix a compact set $K\subset V$ with $(W_1'\cup W_2')\cap V\subset K\subset (U_1\cup U_2)\cap V$.
Let $M_1=K\cap\de U_2, M_2=K\cap\de U_1,$ so that $M_1$ and $M_2$ are disjoint compact subsets of $V$.  This setup is depicted in figure~\ref{fig: 3}.
 \begin{figure}[h]
\begin{center}
\psset{unit=.007in}
\begin{pspicture}(-250,-260)(250,50)
	\newgray{llgray}{.60}
	\pscurve(-250,-120)(-60, -80)(100,-120)(250,-120)(270,-118)
	\pscurve[linecolor=llgray ,linewidth=4pt](-170,-97)(-60,-80)(100,-120)(211,-122)
	\pscurve(-250,-120)(-60, -80)(100,-120)(250,-120)
	\pscurve[linewidth=.75, arrowsize=3pt]{c->}(-220,-180)(-110,-140)(-100,-90)
	\pscurve[linewidth=.75]{c->}(35,-70)(40,-95)(50,-90)(54,-105)
	\pscurve[linewidth=.75]{c->}(15,-120)(0,-100)(-5,-105)(-12,-90)
	\psellipse(-60,-80)(120,120)
	\psellipse(100,-120)(120,120)
	\psset{linestyle=dashed,dash=3pt 2pt}
	\psellipse(-60,-80)(148,148)
	\psellipse(100,-120)(148,148)
	\uput[100](-70,-50){$U_{1}$}
	\uput[0](110,-60){$U_{2}$}
	\uput[0](270,-118){$V$}
	\uput[4](-100,50){$W_{1}$}
	\uput[5](110,5){$W_{2}$}
	\uput[180](-220,-180){$K$}
	\uput[120](45,-80){$M_{2}$}
	\uput[270](20,-110){$M_{1}$}
	\psdots(-15,-87)(55,-110)
\end{pspicture}
\caption{The setup for the local Richberg-type argument.   }
\label{fig: 3}
\end{center}
\end{figure}
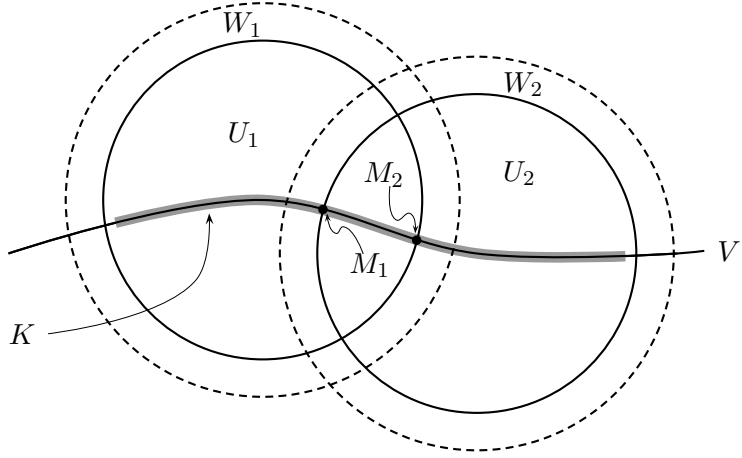

Pick $\theta_1$ a smooth nonnegative cutoff function which is identically $1$ in a neighborhood of $M_2$ in $X$ and
$\theta_2$ a smooth nonnegative cutoff function which is identically $1$ in a neighborhood of $M_1$ in $X$ so that
the supports of $\theta_1$ and $\theta_2$ are disjoint.
Then, if we choose $\eta>0$ small, the functions
\begin{equation}\label{tilda}
\ti{\vp}_j=\vp_j-\eta\theta_j,
\end{equation}
$j=1,2$ have analytic singularities and satisfy $\omega+\ddb \ti{\vp}_j\geq \ve\omega$ on $W_j$.
Furthermore they are smooth in a neighborhood of the generic point of $V\cap W_j$.
On $W_1\cap W_2$ we then define
$$\ti{\vp}_0=\max\{\ti{\vp}_1,\ti{\vp}_2\},$$
which satisfies $\omega+\ddb \ti{\vp}_0\geq \ve\omega$ and equals $\vp$ on $V\cap W_1\cap W_2$.
Consider now a neighborhood of $M_2\backslash P$ in $X$, small enough so that $\theta_1=1$ and $\vp_1,\vp_2$ are finite on it. Since $\vp_1,\vp_2$ agree on $V$ and
are smooth on this neighborhood, we see that there exists a possibly smaller such neighborhood where
\begin{equation*}
\tilde{\phi}_{1} = \phi_{1} - \eta < \phi_{2} = \tilde{\phi}_{2},
\end{equation*}
so that $\ti{\vp}_0=\ti{\vp}_2$ there. Similarly, on any sufficiently small neighborhood of $M_1\backslash P$ we have $\ti{\vp}_0=\ti{\vp}_1$.
Therefore there is an open neighborhood $W_0$ of $(K\backslash ((M_1\cap P)\cup(M_2\cap P)) )\cap \ov{U_1}\cap \ov{U_2}$ in $X$ such that
$\ti{\vp}_0=\ti{\vp}_1$ on $W_0\backslash \ov{U_2}$ and $\ti{\vp}_0=\ti{\vp}_2$ on $W_0\backslash \ov{U_1}$.
Therefore we can define
$$W'=W_0\cup (U_1\backslash \ov{U_2})\cup (U_2\backslash \ov{U_1}),$$
which is a neighborhood of $K\backslash ((M_1\cap P)\cup(M_2\cap P))$ in $X$,
and define a function
$\vp'$ on $W'$ to be equal to $\ti{\vp}_0$ on $W_0$, equal to $\ti{\vp}_1$ on $U_1\backslash \ov{U_2}$ and equal to $\ti{\vp}_2$ on $U_2\backslash \ov{U_1}$.
Then $\vp'$ satisfies $\omega+\ddb\vp'\geq \ve\omega$ and equals $\vp$ on $W'\cap V$. Clearly, $W'$ contains
$(W_1'\cup W_2')\cap (V\backslash P)$. We refer to $W'$ as a pinched neighborhood of $K$, since in general it is not a neighborhood of the whole of $K$ and it might pinch off at points in $M_1\cap P$ and $M_2\cap P$. We will later need to decrease the value of $\eta$, which might change the set $W'$ slightly, but it will still remain a pinched neighborhood of $K$.

We now deal with points in $M_1\cap P$ and $M_2\cap P$. By symmetry, it suffices to consider a point $p\in M_2\cap P$.
Recall that $R = \omega +\ddb F$ is a K\"ahler current on $X$ with
analytic singularities exactly along $V$.  Choose a small coordinate neighborhood $U \subset X$ centered at $p$, small enough so that
$\theta_1=1$ and $\theta_2=0$ on $U$. In particular at points in $W'$ sufficiently near $p$ we have $\vp'=\ti{\vp}_0=\ti{\vp}_2=\vp_2$.
Since $\tilde{\phi}_{1}, \tilde{\phi}_{2}, F$ have
analytic singularities, they can be expressed as (recall that $\vp_1=\vp_2$ on $V\cap U$)
\begin{equation}\label{eq: expansion}
\begin{aligned}
\tilde{\phi}_{1} &= \delta_{1}\log(\sum_{i} |f_{i}|^{2}) + \sigma_{1} - \eta = \phi_{1}-\eta,\\
\tilde{\phi}_{2} & = \delta_{1}\log(\sum_{i} |\tilde{f}_{i}|^{2}) + \sigma_{2} = \phi_{2},\\
F&= \delta_{3}\log(\sum_{j} |g_{j}|^{2}) + \sigma_{3},
\end{aligned}
\end{equation}
near $p$, where $\sigma_{k}$, $k=1,2,3$ are local smooth functions, and $f_i, \ti{f}_i, g_j$ are local holomorphic functions, with the functions $g_j$ locally defining $V$. Moreover,
when restricted to $V$, we have
\begin{equation*}
\delta_{1}\log(\sum_{i} |f_{i}|^{2}) + \sigma_{1} = \delta_{1}\log(\sum_{i} |\tilde{f}_{i}|^{2}) + \sigma_{2}
\end{equation*}
since $\phi_{1}, \phi_{2}$ both extend $\phi$.  Then by the above argument,
the function $\vp'$ is defined at least on the set $\mathcal{S}\subset U$ given by
\begin{equation*}
\mathcal{S} = \left\{\delta_{1}\log(\sum_{i} |\tilde{f}_{i}|^{2})
+ \sigma_{2}> \delta_{1}\log(\sum_{i} |f_{i}|^{2}) + \sigma_{1} - \eta\right\}=\{\ti{\vp}_2>\ti{\vp}_1\}\cap U,
\end{equation*}
where the strict inequality in particular requires $\ti{\vp}_2$ to be finite.
The idea is to show that the singularities of $F$ are comparable to those of
$\tilde{\phi}_{2}$ on $U\backslash \mathcal{S}$.  That is, for $0<\nu <1$
 we consider the subset of $U$ given by
\[\begin{split}
E_\nu &= \left\{\nu\delta_{3}\log(\sum_{j} |g_{j}|^{2}) +
 \nu\sigma_{3} \geq  \delta_{1}\log(\sum_{i} |\tilde{f}_{i}|^{2}) + \sigma_{2} \right\}\\
&=\{\nu F\geq\ti{\vp}_2\}\cap U,
\end{split}\]
where we now allow points where both sides of the inequality are $-\infty$. In particular, we always have that
$p\in E_\nu$. We can also subtract a constant to $F$ so that $\sup_X F\leq 0$, and then we see that
the sets $E_\nu$ are decreasing in $\nu$.

\begin{lem}\label{neigh}
There exists $0<\nu_0<1$ such that for any $0<\nu\leq\nu_0$, the set $E_\nu\, \cup \mathcal{S}$ contains an open
 neighborhood of $p$.
\end{lem}
We refer the reader to figure~\ref{fig: 1} for the geometry of this local gluing problem.
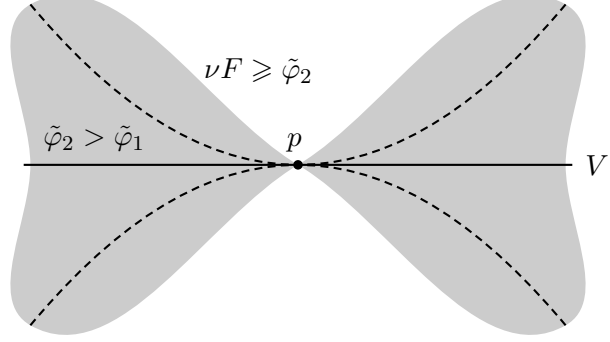
\begin{figure}[h]
\begin{center}
\psset{unit=.007in}
\begin{pspicture}(-250,-260)(250,50)
	\newgray{llgray}{.80}
	\psset{linecolor=llgray}
	\psccurve*(0,-120)(200, 0)(200,-120)(200,-240)(0,-120)(-200, 0)(-200,-120)(-200,-240)
	\psset{linecolor=black}
	\psline(-205,-120)(205,-120)
	\psdot(0,-120)
	\uput[100](0,-120){$p$}
	\uput[0](205,-120){$V$}
	\uput[0](-80,-50){$\nu F \geq \tilde{\phi}_{2}$}
	\uput[0](-200, -100){$\tilde{\phi}_{2}> \tilde{\phi}_{1}$}
	\psset{linestyle=dashed,dash=3pt 2pt}
	\parabola(-200, 0)(0,-120)
	\parabola(-200, -240)(0,-120)
\end{pspicture}
\caption{The geometry of the local gluing problem near a pinched point.  The shaded area corresponds
 to the set $\mathcal{S}$, while the set $E_\nu$ corresponds to the area above
  the upper dashed line, and below the lower dashed line}
  \label{fig: 1}
\end{center}
\end{figure}

The proof of Lemma \ref{neigh} requires several additional lemmas.
The main technique we use is Hironaka's resolution of singularities.
Define an analytic set $G \subset U$
by $$G=\left(V\cup E_+(\omega+\ddb\vp_1)\cup E_+(\omega+\ddb\vp_2)\right)\cap U,$$
and let
$\mathcal{I}_G$ be its defining ideal sheaf.
By shrinking $U$ if necessary, we may assume that every irreducible component of $G$ passes through $p$.
Let $\pi:\tilde{U} \rightarrow U$ be a log resolution of $\mathcal{I}_G$ obtained by blowing up smooth centers. In order to simplify the notation, we assume that we first blow up $V$, to obtain a divisor $D$, and then resolve the strict transform of $G$.  After resolving, we have that
\begin{equation*}
\pi^{-1}(G) = \tilde{V} + \sum_{\ell}D_{\ell}
\end{equation*}
is a sum of smooth divisors with simple normal crossings and $\tilde{V}$ is the irreducible divisor containing $\pi^{-1}(v)$ for a generic point $v \in V$ ($\tilde{V}$ is the strict transform of $D$). If we can show that  $\pi^{-1}(E_\nu \cup \mathcal{S})$ contains an open neighborhood of $\tilde{V} + \sum_{\ell}D_{\ell}$, then it would follow immediately that
$E_\nu\, \cup \mathcal{S}$ contains an open neighborhood of $p$.

\begin{lem}\label{lem: S nbhd}
The set $\pi^{-1}(\mathcal{S})$ contains an open neighborhood of $ \tilde{V}$.
\end{lem}

\begin{proof}
Pick a point $q \in  \tilde{V}$.  Since $\pi$
 is a log resolution, there exists an open set $Z\subset\ti{U}$ with a coordinate
 system $(w_{1},\dots, w_{n})$ centered at $q$ such that $\tilde{V} = \{w_{1}=0\}$ and
  $\pi^{*}\exp(\phi_{1}) , \pi^{*}\exp(\phi_{2})$ are of the form
 \begin{equation*}
\begin{aligned}
\pi^{*}\exp(\phi_{1}) &=U_{1}(w_{1},\dots,w_{n}) \prod_{i=2}^{n} |w_{i}|^{2\alpha_{i}}\\
\pi^{*}\exp(\phi_{2}) &= U_{2}(w_{1},\dots,w_{n})\prod_{i=2}^{n} |w_{i}|^{2\beta_{i}},
\end{aligned}
\end{equation*}
 where $U_{j}$ are smooth, positive functions on $\overline{Z}$, and $\alpha_i, \beta_i$ are nonnegative real numbers.  That $w_{1}$
 does not appear in the product follows from the fact that $\phi_{1}, \phi_{2} \not\equiv -\infty$ on $V\cap U$.
By definition, we have
\begin{equation*}
\pi^{-1}(\mathcal{S})\cap Z = \left\{ w \in Z \bigg| U_{1}(w)\prod_{i \geq 2} |w_{i}|^{2\alpha_{i}}
< e^{\eta} U_{2}(w)\prod_{i\geq 2} |w_{i}|^{2\beta_{i}}\right\}.
\end{equation*}
Now, since $\phi_{1}|_{V} = \phi_{2}|_{V}$, and $w_{i}|_{\ti{V}} \ne 0$ for $i\geq 2$, we clearly have
 that $ \alpha_{i} = \beta_{i}$, and that
 $U_{1}|_{\tilde{V}} = U_{2}|_{\tilde{V}}$.  Since $e^{\eta}>1$, the lemma is proved.
\end{proof}

By Lemma \ref{lem: S nbhd}, it suffices to work on a compact set away
from $\tilde{V}$.  Fix a point $q \in D_{\ell} \cap \overline{\pi^{-1}(\mathcal{S}^{c})}$
and an open set $Z\subset\ti{U}$ disjoint from $\tilde{V}$,  with a coordinate
system $(w_{1},\dots, w_{n})$ centered at $q$ so that
\begin{equation}\label{res form}
\begin{aligned}
\pi^{*}\exp(F) &=U_{F}(w)\prod_{i=1}^{n}|w_{i}|^{2\gamma_{i}},\\
\pi^{*}\exp(\phi_{2}) &=U_{2}(w) \prod_{i=1}^{n} |w_{i}|^{2\beta_{i}}
\end{aligned}
\end{equation}
where $U_{F}(w)$ and $U_{2}(w)$ are smooth, positive functions on $\overline{Z}$, and $\beta_i, \gamma_i$ are nonnegative real numbers.
 Our goal is to find $0<\nu<1$ such that
\begin{equation}\label{descr}
\pi^{-1}(E_\nu) \cap Z = \left\{w\in Z \bigg| U_{2}(w) \prod_{i=1}^{n} |w_{i}|^{2\beta_{i}}
\leq U_{F}^{\nu}(w)\prod_{i=1}^{n}|w_{i}|^{2\nu\gamma_{i}}\right\}
\end{equation}
 contains a neighborhood of $0 \in \mathbb{C}^{n}$.  First, we prove a lemma.

 \begin{lem}\label{crux}
 In equation~\eqref{res form}, if $\gamma_{i}>0$ for some $1\leq i\leq n$, then $\beta_{i}>0$.
 \end{lem}
 \begin{proof}
Suppose for some $i$ we have $\gamma_{i} >0$ but $\beta_{i}=0$.
Let $\{w_{i}=0\} = D_{\ell_{i}}$. Then $\beta_i=0$ means that $\pi^{*}\exp(\phi_{2}) \not\equiv 0$ on $D_{\ell_{i}}$,
while $\gamma_i>0$ means that $\pi^{*}\exp(F)\equiv 0$ on $D_{\ell_i}$.
We have that $\pi(D_{\ell_{i}}) \subset V\cap U$, since $F$ is smooth outside $V$.
 However, we can also expand
 \begin{equation*}
 \pi^{*}\exp(\phi_{1}) =U_{1}(w) \prod_{j=1}^{n} |w_{j}|^{2\alpha_{j}},
 \end{equation*}
 and the fact that $\vp_1\not\equiv-\infty$ on $V\cap U$ again implies that $\alpha_i=0$.
 Since $\phi_{1}|_{V} =\phi_{2}|_{V}$, and $\pi(D_{\ell_{i}})\subset V$, we see again that
 \begin{equation*}
 \pi^{*}\exp(\phi_{1}) = \pi^{*}\exp(\phi_{2})\quad \text{ on } D_{\ell_{i}}.
 \end{equation*}
Recall that $\mathcal{S}=\{\vp_2>\vp_1-\eta\}\cap U$. This shows that $\pi(D_{\ell_i})$ is contained in the interior of $\mathcal{S}$,
which is impossible because we assumed $q\in \overline{\pi^{-1}(\mathcal{S}^{c})}$.
 \end{proof}

With these results in place, we can easily complete the proof of Lemma \ref{neigh}.
\begin{proof}[Proof of Lemma \ref{neigh}]
Thanks to Lemma \ref{crux}, we can choose a small $\nu_{0}>0$ so that $\beta_{i} > \nu_{0} \gamma_{i}$ for each $i$ such that $\gamma_{i}>0$.
 It follows from the description in \eqref{descr} that for any $0<\nu\leq\nu_0$, the set $\pi^{-1}(E_\nu) \cap Z$ contains an open neighborhood of
 the point $q\in D_{\ell} \cap \overline{\pi^{-1}(\mathcal{S}^{c})}$.
 Repeating this finitely many times
   on a covering of $(\sum_{\ell}D_{\ell}) \cap \pi^{-1}(\mathcal{S}^{c})$, and using Lemma \ref{lem: S nbhd},
    we find $\nu_{0}>0$ such that for any $0<\nu\leq\nu_0$, the set
   $\pi^{-1}( E_\nu\cup \mathcal{S} )$  contains an open neighborhood of  $\tilde{V}+ \sum_{\ell}D_{\ell}$. This immediately implies that $E_\nu \cup \mathcal{S}$ contains an open neighborhood of $p$, since $\pi$ is an isomorphism away
    from $\tilde{V}+ \sum_{\ell}D_{\ell}$.
\end{proof}
Furthermore, Lemma \ref{neigh} holds with $\nu_0$ independent of the value of $\eta>0$, which we are then free to decrease later on.
We pick $\nu\leq \nu_0$ so that
 \begin{equation*}
 \omega+ \ddb F \geq \nu \omega,
 \end{equation*}
holds as currents on $X$.
Since $\omega$ is a K\"ahler metric, we have
 \begin{equation*}
 \omega + \ddb \left(\nu F \right) \geq \nu^{2} \omega.
 \end{equation*}
Recall that $E_\nu=\{\nu F\geq \ti{\vp}_2\}\cap U$ and $\mathcal{S}=\{\ti{\vp}_2>\ti{\vp}_1\}\cap U$.
It follows that on $\mathcal{S}$ we have $\ti{\vp}_2>-\infty$, and so
if $U'\subset U$ is a slightly smaller open neighborhood of $p$, then
on $(\de E_\nu)\cap\mathcal{S}\cap U'$ we have $F>-\infty$ and $\nu F=\ti{\vp}_2$. Note that
$$(\de(\mathcal{S}\backslash E_\nu))\cap U'\subset ((\de E_\nu)\cap\mathcal{S}\cap U')\cup(P\cap U'),$$
so $\nu F\geq\vp'$ on $(\de(\mathcal{S}\backslash E_\nu))\cap U'$. This implies that the function
\begin{equation}
   \phi_\nu = \left\{
     \begin{array}{ll}
    \nu F , &  \text{ on }E_\nu\\
       \max\{\phi', \nu F \}, &  \text{ on } \mathcal{S}\backslash E_\nu,
     \end{array}
   \right.
\end{equation}
is defined in a neighborhood $U_p$ of $p$ and satisfies $\omega+\ddb \phi_\nu\geq \nu^2\omega$ (and the value of $\nu$ does not change if we decrease $\eta$, since the
cutoff functions $\theta_1$ and $\theta_2$ are constant on $U$).
Furthermore, since $F$ goes to $-\infty$ on $V$ while $\phi'$ is finite on $(V\backslash P)\cap W'$, we have that
$\phi_\nu$ equals $\phi'=\phi$ on $U_p\cap(V\backslash P)$.

Repeating this argument at every point $p\in M_2\cap P$, as well as every point $p\in M_1\cap P$, taking a finite covering given by the resulting open sets $U_p$,
and taking the smallest $\nu$ of the resulting ones, we conclude that there exists $\nu>0$ sufficiently small such that
$\phi_\nu$ is defined in a whole neighborhood $W_\nu$ of $K$ in $X$, and satisfies the same properties. This completes the first step.

We then fix slightly smaller open sets $W'_\nu\Subset U_\nu\Subset W_\nu$ such that $\cup_{j\geq 3}W_j'\cup W'_\nu$ still covers $V$.
We replace $W_1$ and $W_2$ with $W_\nu$, and replace $\vp_1$ and $\vp_2$ with $\vp_\nu$, and repeat the same procedure with two other open sets in this new covering.
The only difference is that while the functions $\vp_1, \vp_2$ have analytic singularities, this is not the case for $\vp_\nu$, which is instead locally given as the maximum of finitely many functions with analytic singularities. We now explain what modifications are needed in the arguments above.

At any subsequent step, we will have two open sets $W_a, W_b$ with $W_a'\Subset U_a\Subset W_a$ and $W_b'\Subset U_b\Subset W_b$
and with $W'_a\cap W'_b\cap V$ nonempty. On $W_a$ we have a function
$\vp_a$ with $\omega+\ddb\vp_a\geq \ve\omega$, with $\vp_a=\vp$ on $W_a\cap V$, and
similarly for $W_b$. Then exactly as before we obtain a function $\vp'$ on a neighborhood of $K\backslash ((M_a\cap P)\cup(M_b\cap P))$,
which is equal to $\vp'=\max\{\ti{\vp}_a,\ti{\vp}_b\}$ on $W_a\cap W_b$, where we picked cutoff functions $\theta_a, \theta_b$ as before and defined $\ti{\vp}_a=\vp_a-\eta'\theta_a,\ti{\vp}_b=\vp_b-\eta'\theta_b$, where $\eta'>0$ is small enough so that $\omega+\ddb\ti{\vp}_a$ and $\omega+\ddb\ti{\vp}_b$ are larger than $\ve'\omega$ for some $\ve'>0$. Because of the construction we just did, near a point $x\in M_b\cap P$
we can write
$$\vp_a=\max\{\ti{\vp}_{i_1},\dots,\ti{\vp}_{i_p},\nu_a F+\rho_a\},$$
for some $p>0$, some $0<\nu_a<1$ and a continuous function $\rho_a$ (in general the maximum will contain several terms of the form $\nu_j F+\rho_j$, but since $F(x)=-\infty$, up to shrinking
the open set where we work on, only one of them contributes to the maximum, and $\rho_a=\max_j \rho_j$ is continuous). Here the functions $\ti{\vp}_{i_k}$ are defined in \eqref{tilda}, so they have analytic singularities, and so does $F$, and we can also assume that their values of the parameter $\eta$ are all equal and smaller than $\eta'/2$.
Similarly, we can write
$$\vp_b=\max\{\ti{\vp}_{j_1},\dots,\ti{\vp}_{j_q},\nu_b F+\rho_b\}.$$
We work again on a small coordinate neighborhood $U$ centered at $x$ where $\theta_a=1$ and $\theta_b=0$.
We proved earlier that $\vp'$ is defined at least on $\mathcal{S}=\{\ti{\vp}_b>\ti{\vp}_a\}\cap U$. For $0<\nu<1$ we let
$E_\nu=\{\nu F\geq \ti{\vp}_b\}\cap U$. If we can show that there exists $\nu$ such that $E_\nu\cup\mathcal{S}$ contains a neighborhood of $x$, then
we can complete this step exactly as before. For simplicity, we write
$$\hat{\vp}_a=\max\{\ti{\vp}_{i_1},\dots,\ti{\vp}_{i_p}\}, \quad \hat{\vp}_b=\max\{\ti{\vp}_{j_1},\dots,\ti{\vp}_{j_q}\}.$$

To prove this, we pick a log resolution $\pi:\ti{U}\to U$ of the ideal sheaf of
$$G=\left(V\cup \bigcup_{k}E_+(\omega+\ddb\ti{\vp}_{i_k})\cup \bigcup_{\ell}E_+(\omega+\ddb\ti{\vp}_{j_\ell})\right)\cap U,$$
with $\pi^{-1}(G)=\ti{V}+\sum_\ell D_\ell$ as before.
Now note that $\mathcal{S}$ contains the set
$\mathcal{A}\cap\mathcal{B},$ where
$$\mathcal{A}=\{\hat{\vp}_b>\hat{\vp}_a-\eta'\}\cap U,\quad \mathcal{B}=\{\hat{\vp}_b>\nu_a F+\rho_a-\eta'\}\cap U.$$
The set $\mathcal{A}$ equals
$$\left(\bigcup_{\ell=1}^q\bigcap_{k=1}^p \{\ti{\vp}_{j_\ell}>\ti{\vp}_{i_k}-\eta' \}\right)\cap  U.$$
Since $\ti{\vp}_{j_\ell}=\vp_{j_\ell}-\eta\theta_{j_\ell}$ with $0\leq\theta_{j_\ell}\leq 1$, and similarly for $\ti{\vp}_{i_k}$, we see that
$$\{\ti{\vp}_{j_\ell}>\ti{\vp}_{i_k}-\eta' \}\cap U\supset \{\vp_{j_\ell}>\vp_{i_k}-\eta'/2 \}\cap U.$$
Thanks to Lemma \ref{lem: S nbhd}, each of the sets $\pi^{-1}(\{\vp_{j_\ell}>\vp_{i_k}-\eta'/2 \}\cap U)$ contains a neighborhood of $\ti{V}$,
and therefore so does $\pi^{-1}(\mathcal{A})$.
On the other hand, we have that
$E_\nu=\{\nu F\geq \vp_b\}\cap U$ equals
$$\bigcap_{m=1}^q \{\nu F \geq \ti{\vp}_{j_m}\}\cap\{\nu F\geq \nu_b F+\rho_b\}\cap U.$$
If we choose $\nu$ small, then $\{\nu F\geq \nu_b F+\rho_b\}\cap U=U$. Lemma \ref{crux} together with the proof of Lemma \ref{neigh} shows that
there exists $\nu>0$ small such that each set $\pi^{-1}(\{\nu F \geq \ti{\vp}_{j_m}\}\cap U)$ contains a neighborhood of
$$\left(\sum_\ell D_\ell \right)\cap\ov{\pi^{-1}((\cap_{k=1}^p\{\ti{\vp}_{j_m}>
\ti{\vp}_{i_k}-\eta'\}\cap U)^c)}.$$
Therefore, $E_\nu$ contains a neighborhood of
$$\left(\sum_\ell D_\ell \right)\cap\ov{\pi^{-1}((\hat{\vp}_b>
\hat{\vp}_a-\eta'\}\cap U)^c)}=\left(\sum_\ell D_\ell\right) \cap\ov{\pi^{-1}(\mathcal{A}^c)}.$$
This means that $\pi^{-1}(\mathcal{A}\cup E_\nu)$ contains a whole neighborhood of $\ti{V}+\sum_\ell D_\ell$.
On the other hand, the set $E_\nu\cup\mathcal{B}$ equals
$$\left(\{\nu F\geq \hat{\vp}_b\}\cup\{\hat{\vp}_b>\nu_a F+\rho_a-\eta'\}\right)\cap U\supset(V\cup\{\nu F>\nu_a F+\rho_a-\eta'\})\cap U,$$
and if we pick $\nu$ small enough and possibly shrink $U$, then this set equals $U$. This finally proves that $\pi^{-1}(E_\nu\cup\mathcal{S})$ contains a neighborhood of
$\ti{V}+\sum_\ell D_\ell$, which implies that $E_\nu\cup\mathcal{S}$ contains a neighborhood of $x$, and this step is complete.

After at most $N$ such steps, we end up with an open neighborhood $W$ of $V$ in $X$ with a function $\vp''$ defined on $W$ which satisfies
$\omega+\ddb\vp''\geq\ve''\omega$ for some $\ve''>0$, which equals $\vp$ on $V$. Up to shrinking $W$, we may assume that $\vp''$ is defined in a neighborhood of $\de W$.

Now we have a K\"ahler current defined on $W$. On $\de W$ the function $F$ is smooth, so we can choose a large constant $A>0$ such that
$F>\vp''-A$ in a neighborhood of $\de W$.
Therefore we can finally define
\begin{displaymath}
   \Phi = \left\{
     \begin{array}{ll}
     \max\{ \phi'',  F +A\} &  \text{ on }W\\
       F +A, &  \text{ on } X\backslash W,
     \end{array}
   \right.
\end{displaymath}
which is defined on the whole of $X$, it satisfies $\omega+\ddb\Phi\geq \ve'\omega$ for some $\ve'>0$. Since $F$ goes to $-\infty$ on $V$, while $\vp''$ is continuous near the generic point of $V$,
it follows that $\Phi$ equals $\vp$ on $V.$
This completes the proof of Theorem \ref{extend}. \end{proof}

The following simple example illustrates some of the ideas used in the proof of Lemma \ref{neigh}.

 \begin{ex}
Consider $\mathbb{C}^{2}$ with coordinates $(x,y)$.  We take
 $V = \{y=0\} \subset \mathbb{C}^{2}$, and consider the plurisubharmonic function $\phi = \log(|x|^{2})$
 on $V$, and $F = \log(|y|^{2})$ on $\mathbb{C}^{2}$. Take local extensions
\begin{equation*}
\phi_{1} = \log(|x-y|^{2}),\qquad \phi_{2} = \log(|x+y|^{2}),
\end{equation*}
of $\vp$.
We resolve the singularities of $G=\{ (x-y)(x+y)y=0\}$ by the blowup of the origin
$\pi(r,s) = (r,rs)$.  In this case $\{r=0\} = \pi^{-1}(0,0)$, and $\pi^{-1}(G)=\tilde{V}+\sum_{\ell=1}^3 D_\ell$ with
$$\tilde{V} = \{s=0\}, D_1=\{r=0\}, D_2=\{s=1\}, D_3=\{s=-1\}.$$
Then
\begin{equation*}
\pi^{*}(|y|^{2}) = |r|^{2}|s|^{2},\qquad \pi^{*}(|x-y|^{2}) = |r|^{2}|1-s|^{2},
 \qquad \pi^{*}(|x+y|^{2}) = |r|^{2}|1+s|^{2}.
\end{equation*}
In this case, we have
\begin{equation*}
\pi^{-1}(\mathcal{S}) = \{|r|^{2}|1-s|^{2} <e^{\eta}|r|^{2}|1+s|^{2} \} = \{|1-s|^{2} < e^{\eta}|1+s|^{2}\}.
\end{equation*}
Clearly $\pi^{-1}(\mathcal{S})$ contains an open neighborhood of $\tilde{V}\cup \{s=1\}$.
Thus, it suffices to show that the set
\begin{equation*}
\pi^{-1}(E_\nu) = \left\{|r|^{2}|1+s|^{2} \leq |r|^{2\nu}|s|^{2\nu} \right\}
\end{equation*}
contains an open neighborhood of $(\{r=0\} \cup \{s=-1\}) \cap \{|s|>\epsilon\}$.  This is clear,
for any $0<\nu <1$.
\end{ex}

\end{document}